\documentclass[12pt]{amsart}
\usepackage{amsmath,latexsym,amsfonts,amssymb,amsthm}
\usepackage{geometry}
\usepackage{hyperref}
\usepackage{mathrsfs}
\usepackage{graphicx,color}
\usepackage{tikz}
\usepackage{tikz-cd}
\usepackage{float}
\usepackage{enumitem}

\newtheorem{prop}{Proposition}[section]
\newtheorem{thm}[prop]{Theorem}
\newtheorem{cor}[prop]{Corollary}
\newtheorem{conj}[prop]{Conjecture}
\newtheorem{lem}[prop]{Lemma}

\theoremstyle{definition}

\newtheorem{defn}[prop]{Definition}
\newtheorem{expl}[prop]{Example}
\newtheorem{rem}[prop]{\it Remark}
\newtheorem{emp}[prop]{}

\numberwithin{equation}{section}

\newcommand{\bP}{\mathbb{P}}
\newcommand{\bC}{\mathbb{C}}

\newcommand{\bZ}{\mathbb{Z}}
\newcommand{\bF}{\mathbb{F}}
\newcommand{\bD}{\mathbb{D}}
\newcommand{\bQ}{\mathbb{Q}}

\newcommand{\cY}{\mathcal{Y}}
\newcommand{\cC}{\mathcal{C}}
\newcommand{\cO}{\mathcal{O}}

\newcommand{\cS}{\mathcal{S}}
\newcommand{\cA}{\mathcal{A}}

\newcommand{\cD}{\mathcal{D}}

\newcommand{\Supp}{\mathrm{Supp}}
\newcommand{\Sing}{\mathrm{Sing}}

\newcommand{\Bl}{\mathrm{Bl}}
\newcommand{\wt}{\mathrm{wt}}
\newcommand{\rdeg}{\mathrm{rdeg}}

\makeatletter
\def\namedlabel#1#2{\begingroup
   \def\@currentlabel{#2}%
   \label{#1}\endgroup
}
\makeatother

\begin{document}

\title{Construction of hyperbolic Horikawa surfaces}
\author{Yuchen Liu}
\address{Department of Mathematics, Princeton University,
Princeton, NJ, 08544-1000, USA.}
\email{yuchenl@math.princeton.edu}
\date{\today}

\begin{abstract}
 We construct a Brody hyperbolic Horikawa surface that is a double cover of $\bP^2$ branched along a smooth curve of degree $10$. We also construct Brody hyperbolic double covers of Hirzebruch surfaces with branch loci of the lowest possible bidegree.
\end{abstract}

\maketitle

\section{Introduction}

A complex algebraic variety $X$ is said to be \textit{Brody hyperbolic}
if there are no non-constant holomorphic maps from $\bC$ to $X$. Thanks to Brody Lemma \cite{bro78}, we know that a proper Brody hyperbolic variety is \emph{Kobayashi hyperbolic}, i.e. its Kobayashi pseudometric is non-degenerate.  In \cite{lan86},
Lang conjectured that a complex projective variety $X$ is
Brody hyperbolic if every subvariety of $X$ is of general type.
More generally, Green, Griffiths \cite{gg79} and Lang \cite{lan86}
proposed the following conjecture:

\begin{conj}[Green-Griffiths-Lang]\label{ggl}
 If a complex projective variety $X$ is of general type,
 then there exists a proper Zariski closed subset $Z\subsetneq X$
 such that any non-constant holomorphic map $f:\bC\to X$ will satisfy
 $f(\bC)\subset Z$.
\end{conj}

It is easy to see that Lang's conjecture follows from the Green-Griffiths-Lang
conjecture by a Noetherian induction argument. Even in the case of
surfaces, these conjectures are still open.
Based on works of Bogomolov \cite{bog77} and Lu-Yau \cite{ly90},
McQuillan \cite{mcq98} showed that Conjecture \ref{ggl} is
true for minimal surfaces of general type with $c_1^2>c_2$.
Demailly and El Goul \cite{deg00} proved Conjecture \ref{ggl} for some surfaces
with $13c_1^2>9c_2$. In principle, minimal surfaces of general type
with minimal $s_2=c_1^2-c_2$ should be the most difficult case for these
conjectures. For example, a very general quintic surface in $\bP^3$ ($c_1^2=5$, $c_2=55$)
does not contain any rational or elliptic curve by a result of 
Xu \cite{xu94}, but we do not have a single example of 
quintic surface that is Brody hyperbolic.

Recall that the Chern numbers of minimal surfaces of general type 
satisfy the Noether inequality $c_2\leq 5c_1^2+36$.
In the extreme case, a surface that reaches the equality $c_2=5c_1^2+36$
if $c_1^2$ is even and $c_2=5c_1^2+30$ otherwise is called a \textit{Horikawa
surface}. A Horikawa surface with even $c_1^2$
is classified to be either a double cover of $\bP^2$ or of
a Hirzebruch surface (see \cite{hor76}). For instance, a double cover of $\bP^2$
branched along a smooth curve of degree $10$ is Horikawa.
Using orbifold techniques, Roulleau and Rousseau \cite{rr13}
showed that a very general
member of this class of Horikawa surfaces is algebraic 
hyperbolic (in particular it has no rational or 
elliptic curve). Hence a very general
member of this class of surfaces is expected to be
Brody hyperbolic according to Conjecture \ref{ggl}.

Our first main result shows that there exists 
a Horikawa surface in this class that is Brody hyperbolic.
This gives an analytic generalization of Roulleau-Rousseau's result (in particular implies \cite[Theorem 3.2]{rr13}) and also provides evidence supporting Conjecture \ref{ggl}. 

\begin{thm}\label{main}
 Let $d$ be an even integer. Then there exists a smooth plane 
 curve $D$ of degree
 $d$ such that the double cover of $\bP^2$ branched along $D$ is Brody hyperbolic
 if and only if $d\geq 10$.
\end{thm}

We remark that here that some Brody hyperbolic
double covers of $\bP^2$ have been constructed in 
\cite[Theorem 5]{liu16} with branch loci of minimal 
degree $30$.

For an integer $N\geq 0$, let $\bF_N$ be the $N$-th Hirzebruch surface.
The surface $\bF_N$ has a natural fibration $\bF_N\to \bP^1$. Denote
by $F$ a fiber, and by $T$ a section of the fibration such that
$(T^2)=N$. Any divisor $D$ on $\bF_N$ is linearly equivalent
to $aF+bT$ for integers $a$ and $b$, and we say that $D$ is of bidegree $(a,b)$.

In \cite{rr13}, Roulleau and Rousseau also showed that a very general
Horikawa surface that is a double cover of $\bF_N$ branched along a curve of bidegree $(6,6)$ does not contain
a rational curve. In general it will contain an elliptic curve, so it
cannot be Brody hyperbolic. In the next theorem, we construct
smooth curves of the lowest possible bidegrees in $\bF_N$
along which the double covers of $\bF_N$ are Brody hyperbolic.

\begin{thm}\label{main2}
 Let $a,b$ be even integers.
 Then there exists a smooth curve $D\subset\bF_N$ in the 
 linear system $|aF+bT|$ such that 
 the double cover of $\bF_N$ branched along $D$ is Brody hyperbolic if and only if
 one of the following is true:
 \begin{itemize}
  \item $N=0$ and $a,b\geq 8$;
  \item $N\geq 1$, $a\geq 6$ and $b\geq 8$.
 \end{itemize}
\end{thm}

The ``only if'' parts of Theorem \ref{main} and \ref{main2} are somewhat easy which follow by showing the existence of a rational or elliptic curve on the double cover when the branch locus has a smaller (bi)degree.

Our strategy to prove the ``if'' parts of Theorem 
\ref{main} and \ref{main2} is by using a 
degeneration process consisting of three steps.
Denote by $X$ the base surface $\bP^2$ or $\bF_N$.
In step 1, we degenerate the branch locus $D$ to 
a non-reduced double curve $2C$ where $C$ is smooth.
As a result, the double cover degenerates to a 
union of two copies of $X$ glued along $C$. Using
stability of intersections of entire curves, it
suffices to show that both $X\setminus(C\setminus D)$
and $C$ are Brody hyperbolic. In step 2, we 
degenerate $C$ into a line arrangement $\cup_i C_i$. By a variant of Zaidenberg's method \cite{zai89}, it suffices to show that $X\setminus((\cup_i C_i)\setminus D)$ is Brody hyperbolic. By classical results, we know that for $X=\bP^2$ or $\bF_N$, $X\setminus(\cup_i C_i)$ is Brody hyperbolic. In step 3, we apply Zaidenberg-Duval's method \cite{zai89, sz02, duv04, duv} of degenerating $D$ into line arrangements in order to deduce  hyperbolicity of $X\setminus((\cup_i C_i)\setminus D)$ from hyperbolicity of $X\setminus(\cup_i C_i)$ which is known by classical results. 
\smallskip

\emph{Historical remark.} Note that Duval \cite{duv04} constructed a Brody hyperbolic sextic surface in $\bP^3$ by nicely adopting Zaidenberg's method \cite{zai89}, together with the hyperbolic non-percolation introduced in \cite{sz02}.
In this paper, we follow precisely Duval's approach \cite{duv04} which was further developed in \cite{huy16}. Thus we will use the term \emph{Zaidenberg-Duval's method} for this approach in our presentation.
\medskip

The paper is organized as follows. In Section \ref{duval},
we recall Zaidenberg-Duval's method \cite{zai89, sz02, duv04,  duv} in constructing a smooth curve $D$ satisfying the hyperbolicity of $X\setminus((\cup_i C_i)\setminus D)$.
We recall results in \cite[Section 4]{huy16} in the surface case in Lemma \ref{huynh}, and we apply this lemma to $\bP^2$ and $\bF_N$ in Corollary \ref{huy_p2} and \ref{huy_fn}. In order to deform $\cup_i C_i$ into a smooth curve $C$ preserving the hyperbolicity of $X\setminus (C\setminus D)$, we apply Zaidenberg's method \cite{zai89}
in Section \ref{zai_sec} (see Lemma \ref{zai}). Starting with a log smooth projective surface pair $(X,D)$ and a set of rational curves
$\{C_i\}$ with $X\setminus((\cup_i C_i)\setminus D)$ being Brody hyperbolic, we introduce the concept of \emph{admissible deformation} (see Definition \ref{adm_def}) in order to preserve the hyperbolicity of $X\setminus (C\setminus D)$ under deformation. Using the technique of smoothing of rational trees in the deformation process (e.g. \cite[II.7]{kol96}), we are able to translate an admissible deformation of rational curves into an admissible contraction of their dual graphs (see Lemma \ref{primitive}). In Section \ref{graph_sec}, we study dual graphs that can be admissibly contracted into singletons. Using these results, we construct a smoothing $C$ of $\cup_{i}C_i$ preserving the  hyperbolicity of $X\setminus(C\setminus\Delta)$ under certain assumptions on the dual graph of $\cup_i C_i$ (see Lemma \ref{zai_gen}). Applying this lemma to $X=\bP^2$ or $\bF_N$ gives smooth curves $C$ and $D$ with certain (bi)degrees such that $X\setminus(C\setminus D)$ is Brody hyperbolic. In Section \ref{proofs}, we prove Theorem \ref{main} and \ref{main2}. As an application of Theorem \ref{main}, we give new examples of Brody hyperbolic surfaces in $\bP^3$ of minimal degree $10$ that are cyclic covers of $\bP^2$ under linear projections (see Theorem \ref{cycsurf}). This also improves \cite[Theorem 25]{liu16}. We mention that a Brody hyperbolic Horikawa surface of even $c_1^2$ has to be a double cover of $\bP^2$ branched along a degree $10$ curve (see Remark \ref{finalrmk}).

\subsection*{Notation}
Throughout this paper, we work over the complex numbers $\bC$.
For a subset $U$ of a projective variety $X$, we say that
$U$ is \emph{Brody hyperbolic} if any non-constant holomorphic map
$\phi:\bC\to X$ satisfies $\phi(\bC)\not\subset U$. A divisor $D$ on a smooth surface $X$ is \emph{normal crossing} if $D$ is reduced and has only nodal singularities. Moreover, a normal crossing divisor $D$ is said to be \emph{simple normal crossing} if all irreducible components of $D$ are smooth. We say that a surface pair $(X,D)$ is \emph{log smooth} if $X$ is a smooth surface and $D$ is a simple normal crossing divisor on $X$. A reduced projective curve is \emph{stable} (in the sense of Deligne-Mumford) if it has only nodal singularities and its dualizing sheaf is ample.

\subsection*{Acknowledgement}
I would like to thank Xavier Roulleau and Erwan Rousseau for fruitful discussions. I wish to thank Dinh Tuan Huynh, J\'anos Koll\'ar and Ziquan Zhuang for helpful comments and suggestions, and Christian Liedtke for his interest. I also wish to thank the anonymous referees for their careful work. The author was partially supported by NSF Grant DMS-1362960.

\section{Zaidenberg-Duval's method}\label{duval}

We first recall the following known facts from complex analysis whose proof is a simple application of the classical Hurwitz Theorem.
(See also \cite[3.6.11]{kob98}, \cite[Stability of intersections]{duv} or \cite[Section 3.1]{huy16}.)

\begin{lem}[Stability of intersections]\label{stabint}
 Let $X$ be a normal proper complex analytic space. Let $S$ be an effective
 Weil divisor in $X$, i.e. $S$ is a sum of closed analytic
 subvarieties of codimension $1$. Suppose that a sequence of 
 entire curves $(\phi_n)$ in $X$ converges to an entire 
 curve $\phi$. If $\phi(\bC)\not\subset \Supp(S)$, then 
 \[
  \phi(\bC)\cap S^{\circ}\subset\lim_{n\to\infty}
  \phi_n(\bC)\cap \Supp(S),
 \]
 where $S^{\circ}:=\{x\in\Supp(S)\mid S\textrm{ is }\bQ\textrm{-Cartier in a neighborhood of }
 x\}$.
\end{lem}

The following lemma was proved in \cite[Section 4]{huy16} (see also \cite[Lemma]{duv}).

\begin{lem}\label{huynh}
 Let $X$ be a smooth projective surface. Let $\{C_i\}_{i=1}^m$ be a set of irreducible curves on $X$ such that $(X,\sum_{i=1}^m C_i)$ is log smooth. Let $L$ be a globally generated line bundle on $X$. Assume the following holds:
 \begin{enumerate}[label=(\alph*)]
     \item $X\setminus (\cup_{i=1}^m C_i)$ is Brody hyperbolic;
     \item For any $i$, $\cup_{j\neq i} C_j$ is a stable curve;
     \item For any $i$, there exists an effective Cartier divisor $H_i\in |L|$ such that $\Supp(H_i)=\cup_{j\neq i} C_j$.
 \end{enumerate}
 Then there exists a smooth curve $S\in |L|$ such that $(X, S+\sum_{i=1}^m C_i)$ is log smooth
 and $X\setminus \big((\cup_{i=1}^m C_i)\setminus S\big)$ is Brody hyperbolic.
\end{lem}

\begin{proof} See \cite[Section 4]{huy16}.
 \end{proof}

 The following corollary was proved in \cite[Section 4]{huy16} using Lemma \ref{huynh}.

\begin{cor}[{\cite[Section 4]{huy16}}]\label{huy_p2}
 Let $\{C_i\}_{i=1}^{m}$ be a set of lines in general position in
 $\bP^2$ with $m\geq 5$. Let $d\geq 4$ be an integer.
 Then there exists a smooth plane curve $S$ of degree $d$ such that $(\bP^2, S+\sum_{i=1}^m C_i)$ is log smooth
 and $\bP^2\setminus \big((\cup_{i=1}^m C_i) \setminus S\big)$ is Brody hyperbolic.
\end{cor}

\begin{cor}\label{huy_fn}
Let $\{C_i\}_{i=1}^{a+b}$ be a set of curves in $\bF_N$. 
Assume that $C_i$ is a general curve in $|F|$ for 
any $i\leq a$; $C_j$ is a general curve in $|T|$
for any $j>a$. Then 
there exists a smooth curve $S\in |cF+dT|$ in $\bF_N$
such that $(\bF_N, S+\sum_{i=1}^{a+b}C_i)$ is log smooth and $\bF_N\setminus\big((\cup_{i=1}^{a+b}C_i)\setminus S\big)$ is Brody hyperbolic if one of the
following is true:
\begin{itemize}
 \item $N=0$ and $a,b,c,d\geq 4$;
 \item $N\geq 1$, $a,c\geq 3$ and $b,d\geq 4$.
\end{itemize}
\end{cor}

\begin{proof}
 Firstly, let us consider special cases where $a, b$ achieve their minima, i.e. $N=0$, $a=b=4$ or $N\geq 1$, $a=3$, $b=4$.
 Since both linear systems $|F|$ and $|T|$ are base point free, for a general choice of $\{C_i\}_{i=1}^{a+b}$ the pair $(\bF_N, \sum_{i=1}^{a+b}C_i)$ is log smooth.
 Let $L:=\cO_{\bF_N}(cF+dT)$ be a line bundle on $\bF_N$.
 Then we only need to show that the assumptions (a), (b) and (c) of Lemma \ref{huynh} are fulfilled for $(\bF_N,\sum_{i=1}^{a+b}C_i)$ and $L$.
 
 If $N=0$ and $a=b=4$, then $\bF_0=\bP^1\times
 \bP^1$ and $\{C_i\}_{i=1}^8$ consists of $4$ 
 vertical lines and $4$ horizontal lines in 
 general position. It is clear that 
 $\bF_0\setminus\cup_{i=1}^8 C_i\cong 
 (\bP^1\setminus \{4\textrm{ points}\})\times
 (\bP^1\setminus \{4\textrm{ points}\})$ is Brody
 hyperbolic, so (a) is satisfied. For (b), each
 $C_j$ intersects four $C_k$'s with $k\neq j$. So
 for any $i\neq j$, $C_j$ intersects with at least
 three $C_k$'s with $k\not\in\{i,j\}$.
 Since $C_i\cap C_j\cap C_k=\emptyset$,
 $\cup_{j\neq i}C_j$ is stable, hence (b) is satisfied.
 Since $c,d\geq 4$, $C_i\sim F$ for $1\leq i\leq 4$ and 
 $C_j\sim T$ for $5\leq j\leq 8$, it is easy to see that (c) is also satisfied.
 
 If $N\geq 1$, $a=3$ and $b=4$, then the natural fibration $\pi: \bF_N\to \bP^1$ maps $C_1$, $C_2$, $C_3$ to three distinct points in $\bP^1$. It is clear that $(F\cdot C_i)=(F\cdot T)=1$ for $i=4,\cdots,7$. Hence for a general
 choice of $\{C_i\}_{i=4}^7$, the set $F\cap(\cup_{i=4}^7 C_i)$ has at least three points for any fiber $F$ of $\pi$. Since $\bP^1\setminus\{3\textrm{ points}\}$ is Brody hyperbolic, the
 fiber and the base of $\pi:\bF_N\setminus\cup_{i=1}^7 C_i\to\bP^1\setminus\{\pi(C_1),\pi(C_2),\pi(C_3)\}$ are Brody hyperbolic. Hence (a) is satisfied. For (b), each $C_i$ with $1\leq i\leq 3$ intersects each $C_k$ with $4\leq k\leq 7$. Since $(T^2)=N\geq 1$, each $C_i$ with $4\leq i\leq 7$ intersects 
 each $C_k$ with $k\neq i$. As a result, each $C_i$ intersects with at least four $C_k$'s with $k\neq i$. So (b) is satisfied by the same reason as in the last paragraph. Since $c\geq 3$, $d\geq 4$, $C_i\sim F$ for $1\leq i\leq 3$, and $C_i\sim T$ for $4\leq i\leq 7$, it is easy to see that (c) is also satisfied.
 
 Up to now we have shown the corollary for cases where $a,b$ achieve
 their minima. More precisely, under the assumptions of
 $N,a,b,c,d$, for general choices of
 $\{C_i\}_{i=1}^{a_{\min}}$ and $\{C_j\}_{j=a+1}^{a+b_{\min}}$
 there exists a smooth curve $S\in |cF+dT|$ in $\bF_N$, such
 that $(\bF_{N}, S+\sum_{i=1}^{a_{\min}}C_i+\sum_{j=a+1}^{a+b_{\min}} C_j)$ is log smooth
 and $\bF_N\setminus\big(((\cup_{i=1}^{a_{\min}}C_i)\cup(\cup_{j=a+1}^{a+b_{\min}} C_j) )\setminus S\big)$
 is Brody hyperbolic. 
 If one of $a,b$ is strictly bigger than its minimum, then 
 \[
 \bF_N\setminus\big((\cup_{i=1}^{a+b}C_i)\setminus S\big)\subset
 \bF_N\setminus\big(((\cup_{i=1}^{a_{\min}}C_i)\cup(\cup_{j=a+1}^{a+b_{\min}} C_j) )\setminus S\big)
 \]
 where the latter set is Brody hyperbolic.
 Hence $\bF_N\setminus\big((\cup_{i=1}^{a+b}C_i)\setminus S\big)$
 is Brody hyperbolic. Besides, since $(\bF_{N}, S+
 \sum_{i=1}^{a_{\min}}C_i+\sum_{j=a+1}^{a+b_{\min}} C_j)$ is log smooth, for general choices of $\{C_i\}_{i=a_{\min}+1}^a$ and $\{C_j\}_{j=a+b_{\min}+1}^{a+b}$ we also have that $(X, S+\sum_{i=1}^{a+b} C_i)$ is log smooth. This finishes the proof.
\end{proof}

\section{Admissible contractions of multigraphs}\label{graph_sec}

\begin{defn}
\begin{enumerate}[label=(\alph*)]
\item A \emph{vertex-weighted multigraph} $G$ is an ordered quadruple $(V,E,r, \wt)$ such that
\begin{itemize}
\item $V$ is a finite set of vertices;
\item $E$ is a finite set of edges;
\item $r: E\to \{\{v,w\}:v,w\in V, v\neq w\}$ assigns each edge an
unordered pair of endpoint vertices;
\item $\wt:V\to \bZ$ assigns to each vertex an integer as its weight.
\end{itemize}
\item For a vertex $v\in V$, we define the \emph{degree} (respectively \emph{reduced degree}) of $v$ to be its number of incident edges (respectively adjacent vertices). More precisely, 
\[
\deg(v):=\#\{e\in E: v\in e\},\quad \rdeg(v):=\#\{w\in V: \{v,w\}\in r(E)\}.
\]
\item Let $G_1,G_2$ be two vertex-weighted multigraphs. We say that
$G_1$ is a \emph{submultigraph} of $G_2$ if there exists injective maps
$\phi: V_1\to V_2$ and $\psi: E_1\to E_2$, such that
$r_2\circ \psi = \phi\circ r_1$ and $\wt_1\leq \wt_2\circ\phi$.
If, moreover, $\phi$ is bijective, then we say that $G_1$ is a \emph{spanning submultigraph} of $G_2$.

\item A vertex-weighted multigraph is \emph{completely multipartite} if there does not exist a triple of vertices $\{v_1,v_2,v_3\}$ such that both $\{v_1,v_2\}$ and $\{v_1,v_3\}$ are non-adjacent, but $\{v_2,v_3\}$ is adjacent.
\end{enumerate}
\end{defn}

\begin{defn}
\begin{enumerate}[label=(\alph*)]
\item
Let $G, G'$ be two vertex-weighted multigraphs. We say that $G'$ is a
\emph{contraction} of $G$ \emph{with respect to} a pair of adjacent vertices $\{v,w\}$ 
in $G$ if there exist maps $\phi: V\to V'$ and $\psi: E\setminus r^{-1}(\{v,w\})\to E'$ such that
\begin{itemize}
 \item $\phi(v)=\phi(w)$, and $\phi$ induces a bijection between $V\setminus\{v,w\}$ and $V'\setminus\{\phi(v)\}$;
 \item $\psi$ is bijective, and $r'\circ\psi=\phi\circ r$ as maps from
 $E\setminus r^{-1}(\{v,w\})$;
 \item $\wt'(\phi(v))=\wt(v)+\wt(w)$, and $\wt'\circ\phi =\wt$ as maps
 from $V\setminus\{v,w\}$ to $\bZ$.
\end{itemize}
\item A contraction $G'$ of $G$ with respect to $\{v,w\}$
is said to be \emph{admissible} if there exists a non-negative integer $l <\# r^{-1}(v,w)$ such that 
the following conditions hold:
\begin{itemize}
 \item For each vertex $x$ other than $v$ and $w$, $\deg(x)\geq 3$;
 \item $\wt(v)\geq l+1$ and $\wt(w)\geq l+2$;
 \item $\deg(v)-\# r^{-1}(\{v,w\})+l\geq 3$ and $\deg(w)-\#  r^{-1}(\{v,w\})+l\geq 3$.
\end{itemize}
\item A vertex-weighted multigraph $G$ is said to be \emph{admissibly contractible} if there exists a sequence of vertex-weighted multigraphs
$(G_i)_{i=0}^k$ such that $G_0=G$, $G_k$ is a singleton, and $G_i$ is an
admissible contraction of $G_{i-1}$ for each $1\leq i\leq k$.
\end{enumerate}
\end{defn}

\begin{expl}
We give an illustration of an admissible contraction of vertex-weighted multigraphs.
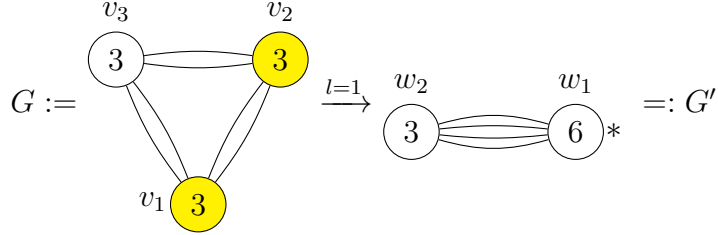
\begin{center}
\qquad $G:=$
\begin{tikzpicture}[scale=0.8, baseline=(current bounding box.center)]
\node[circle,fill=yellow,draw,label={[shift={(-0.6,-0.6)}]$v_1$}] (a) at (1.85,0){$3$};
\node[circle,fill=yellow,draw,label={$v_2$}] (c) at (3.2,2.4){$3$};
\node[circle,draw,label={$v_3$}] (e) at (0.5,2.4){$3$};

\draw (a) edge[bend left=8] (c);
\draw (a) edge[bend left=8] (e);
\draw (a) edge[bend right=8] (c);
\draw (a) edge[bend right=8] (e);
\draw (c) edge[bend left=8] (e);
\draw (c) edge[bend right=8] (e);
\end{tikzpicture}
$\xrightarrow{l=1}$
\begin{tikzpicture}[scale=0.8, baseline=(current bounding box.center)]
\node[circle,draw,label={$w_2$}] (c) at (0,0){$3$};
\node[circle,draw,label={$w_1$},label={[shift={(0.5,-0.6)}]$*$}] (e) at (2.7,0){$6$};

\draw (c) edge[bend left=6] (e);
\draw (c) edge[bend right=6] (e);
\draw (c) edge[bend left=16] (e);
\draw (c) edge[bend right=16] (e);
\end{tikzpicture}
$=:G'$
\end{center}

Here $V_G=\{v_1,v_2,v_3\}$, $V_{G'}=\{w_1,w_2\}$, $\wt_G(v_i)=3$ for any $1\leq i\leq 3$, $\wt_{G'}(w_1)=6$ and $\wt_{G'}(w_2)=3$. Each vertex is represented as a circle in the picture. The name of each vertex is marked outside the circle, and its weight is marked inside the circle. Each edge connecting two vertices is represented as an arc connecting two circles.

In the illustration above, we see that $H$ is a contraction of $G$ with
respect to $\{v_1,v_2\}$, where $\phi$ is given by $\phi(v_1)=\phi(v_2)=w_1$ and $\phi(v_3)=w_2$. Each contraction is represented as an arrow. The two merging vertices of $G$ are represented as yellow filled circles, and we mark $*$ outside circles representing their images under $\phi$.
If a contraction is admissible, we mark the corresponding value of $l$ above the arrow. It is easy to verify that in the picture above, $G'$ is an admissible contraction of $G$ with respect to $\{v_1,v_2\}$.
\end{expl}

The following lemma follows easily from the definitions.
\begin{lem}
 Let $G$ be a vertex-weighted multigraph. Let $H$ be a spanning submultigraph of $G$. If $H$ is admissibly contractible, then so is $G$. 
\end{lem}

\begin{prop}\label{samplegraphs}
The following vertex-weighted multigraphs $K_1$, $K_2$, $K_3$ and $K_4$ are all admissibly contractible:
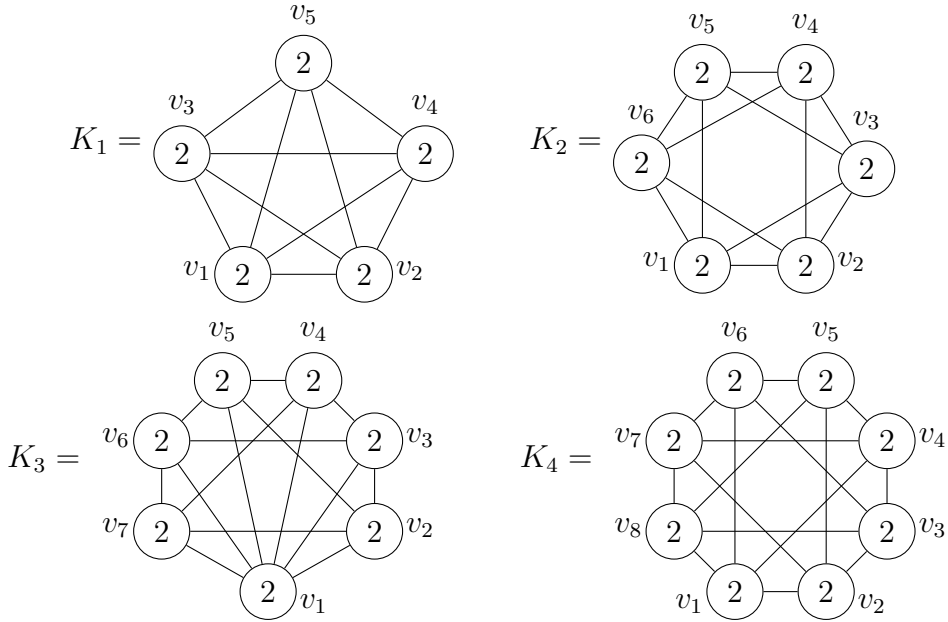
\begin{center}
$K_1=$
\begin{tikzpicture}[scale=0.8, baseline=(current bounding box.center)]
\node[circle,draw,label={[shift={(-0.6,-0.6)}]$v_1$}] (a) at (1,0){$2$};
\node[circle,draw,label={[shift={(0.6,-0.6)}]$v_2$}] (b) at (3,0){$2$};
\node[circle,draw,label={$v_3$}] (c) at (0,2){$2$};
\node[circle,draw,label={$v_4$}] (d) at (4,2){$2$};
\node[circle,draw,label={$v_5$}] (e) at (2, 3.5){$2$};
\draw (a) edge (b);
\draw (a) edge (c);
\draw (a) edge (d);
\draw (a) edge (e);
\draw (b) edge (c);
\draw (b) edge (d);
\draw (b) edge (e);
\draw (c) edge (d);
\draw (c) edge (e);
\draw (d) edge (e);
\end{tikzpicture}
\qquad
$K_2=$
\begin{tikzpicture}[scale=0.8, baseline=(current bounding box.center)]
\node[circle,draw,label={[shift={(-0.6,-0.6)}]$v_1$}] (a) at (1,0){$2$};
\node[circle,draw,label={[shift={(0.6,-0.6)}]$v_2$}] (b) at (2.7,0){$2$};
\node[circle,draw,label={$v_3$}] (c) at (3.7,1.6){$2$};
\node[circle,draw,label={$v_4$}] (d) at (2.7,3.2){$2$};
\node[circle,draw,label={$v_5$}] (e) at (1,3.2){$2$};
\node[circle,draw,label={$v_6$}] (f) at (0,1.7){$2$};
\draw (a) edge (b);
\draw (a) edge (c);
\draw (a) edge (e);
\draw (a) edge (f);
\draw (b) edge (c);
\draw (b) edge (d);
\draw (b) edge (f);
\draw (c) edge (d);
\draw (c) edge (e);
\draw (d) edge (e);
\draw (d) edge (f);
\draw (e) edge (f);
\end{tikzpicture}
\end{center}

\begin{center}
$K_3=$
\begin{tikzpicture}[scale=0.8, baseline=(current bounding box.center)]
\node[circle,draw,label={[shift={(0.6,-0.8)}]$v_1$}] (a) at (1.75,0){$2$};
\node[circle,draw,label={[shift={(-0.6,-0.6)}]$v_7$}] (h) at (0,1){$2$};
\node[circle,draw,label={[shift={(-0.6,-0.6)}]$v_6$}] (g) at (0,2.5){$2$};
\node[circle,draw,label={[shift={(0.6,-0.6)}]$v_2$}] (c) at (3.5,1){$2$};
\node[circle,draw,label={[shift={(0.6,-0.6)}]$v_3$}] (d) at (3.5,2.5){$2$};
\node[circle,draw,label={$v_5$}] (f) at (1,3.5){$2$};
\node[circle,draw,label={$v_4$}] (e) at (2.5,3.5){$2$};
\draw (a) edge (d);
\draw (a) edge (f);
\draw (a) edge (h);
\draw (a) edge (c);
\draw (a) edge (e);
\draw (a) edge (g);
\draw (c) edge (d);
\draw (c) edge (f);
\draw (c) edge (h);
\draw (d) edge (e);
\draw (d) edge (g);
\draw (e) edge (f);
\draw (e) edge (h);
\draw (f) edge (g);
\draw (g) edge (h);
\end{tikzpicture}
\qquad
$K_4=$
\begin{tikzpicture}[scale=0.8, baseline=(current bounding box.center)]
\node[circle,draw,label={[shift={(-0.6,-0.8)}]$v_1$}] (a) at (1,0){$2$};
\node[circle,draw,label={[shift={(0.6,-0.8)}]$v_2$}] (b) at (2.5,0){$2$};
\node[circle,draw,label={[shift={(-0.6,-0.6)}]$v_8$}] (h) at (0,1){$2$};
\node[circle,draw,label={[shift={(-0.6,-0.6)}]$v_7$}] (g) at (0,2.5){$2$};
\node[circle,draw,label={[shift={(0.6,-0.6)}]$v_3$}] (c) at (3.5,1){$2$};
\node[circle,draw,label={[shift={(0.6,-0.6)}]$v_4$}] (d) at (3.5,2.5){$2$};
\node[circle,draw,label={$v_6$}] (f) at (1,3.5){$2$};
\node[circle,draw,label={$v_5$}] (e) at (2.5,3.5){$2$};
\draw (a) edge (b);
\draw (a) edge (d);
\draw (a) edge (f);
\draw (a) edge (h);
\draw (b) edge (c);
\draw (b) edge (e);
\draw (b) edge (g);
\draw (c) edge (d);
\draw (c) edge (f);
\draw (c) edge (h);
\draw (d) edge (e);
\draw (d) edge (g);
\draw (e) edge (f);
\draw (e) edge (h);
\draw (f) edge (g);
\draw (g) edge (h);
\end{tikzpicture}
\end{center}

\end{prop}

\begin{proof}
For simplicity, we will omit the name of vertices in all pictures.
 A successive admissible contraction of $K_1$ is illustrated as below:
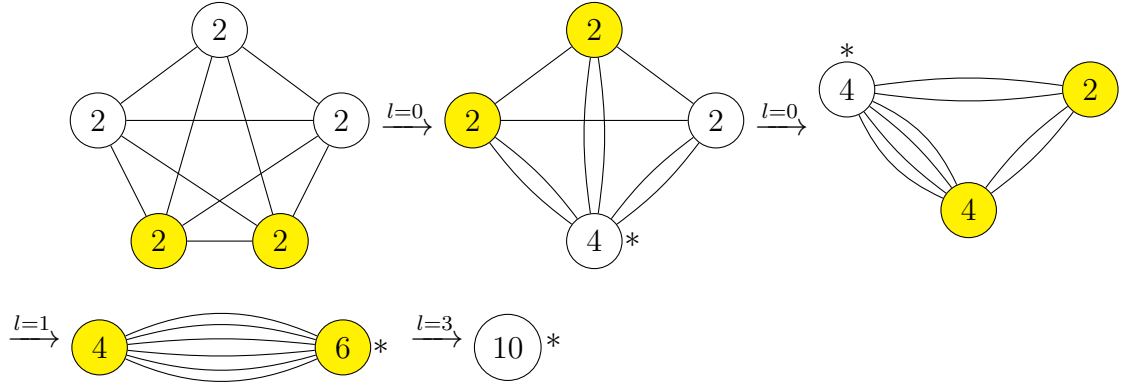
\begin{flushleft}
\qquad
\begin{tikzpicture}[scale=0.8, baseline=(current bounding box.center)]
\node[circle,fill=yellow,draw] (a) at (1,0){$2$};
\node[circle,fill=yellow,draw] (b) at (3,0){$2$};
\node[circle,draw] (c) at (0,2){$2$};
\node[circle,draw] (d) at (4,2){$2$};
\node[circle,draw] (e) at (2, 3.5){$2$};
\draw (a) edge (b);
\draw (a) edge (c);
\draw (a) edge (d);
\draw (a) edge (e);
\draw (b) edge (c);
\draw (b) edge (d);
\draw (b) edge (e);
\draw (c) edge (d);
\draw (c) edge (e);
\draw (d) edge (e);
\end{tikzpicture}
$\xrightarrow{l=0}$
\begin{tikzpicture}[scale=0.8, baseline=(current bounding box.center)]
\node[circle,draw,label={[shift={(0.5,-0.6)}]$*$}] (a) at (2,0){$4$};
\node[circle,fill=yellow,draw] (c) at (0,2){$2$};
\node[circle,draw] (d) at (4,2){$2$};
\node[circle,fill=yellow,draw] (e) at (2, 3.5){$2$};
\draw (a) edge[bend left=8] (c);
\draw (a) edge[bend right=8] (c);
\draw (a) edge[bend left=8] (d);
\draw (a) edge[bend right=8] (d);
\draw (a) edge[bend left=8] (e);
\draw (a) edge[bend right=8] (e);
\draw (c) edge (d);
\draw (c) edge (e);
\draw (d) edge (e);
\end{tikzpicture}
$\xrightarrow{l=0}$
\begin{tikzpicture}[scale=0.8, baseline=(current bounding box.center)]
\node[circle,fill=yellow,draw] (a) at (2,0){$4$};
\node[circle,draw,label={[shift={(0,-0.1)}]$*$}] (c) at (0,2){$4$};
\node[circle,fill=yellow,draw] (d) at (4,2){$2$};
\draw (a) edge[bend left=8] (c);
\draw (a) edge[bend right=8] (c);
\draw (a) edge[bend left=20] (c);
\draw (a) edge[bend right=20] (c);
\draw (a) edge[bend left=8] (d);
\draw (a) edge[bend right=8] (d);
\draw (c) edge[bend left=8] (d);
\draw (c) edge[bend right=8] (d);
\end{tikzpicture}
\end{flushleft}

\begin{flushleft}
$\xrightarrow{l=1}$
\begin{tikzpicture}[scale=0.8, baseline=(current bounding box.center)]
\node[circle,fill=yellow,draw] (c) at (0,0){$4$};
\node[circle,fill=yellow,draw,label={[shift={(0.5,-0.6)}]$*$}] (d) at (4,0){$6$};

\draw (c) edge[bend left=6] (d);
\draw (c) edge[bend right=6] (d);
\draw (c) edge[bend left=16] (d);
\draw (c) edge[bend right=16] (d);
\draw (c) edge[bend left=24] (d);
\draw (c) edge[bend right=24] (d);
\end{tikzpicture}
$\xrightarrow{l=3}$
\begin{tikzpicture}[scale=0.8, baseline=(current bounding box.center)]
\node[circle,draw,label={[shift={(0.6,-0.6)}]$*$}] (c) at (0,0){$10$};
\end{tikzpicture}
\end{flushleft}
\medskip

 A successive admissible contraction of $K_2$ is illustrated as below:
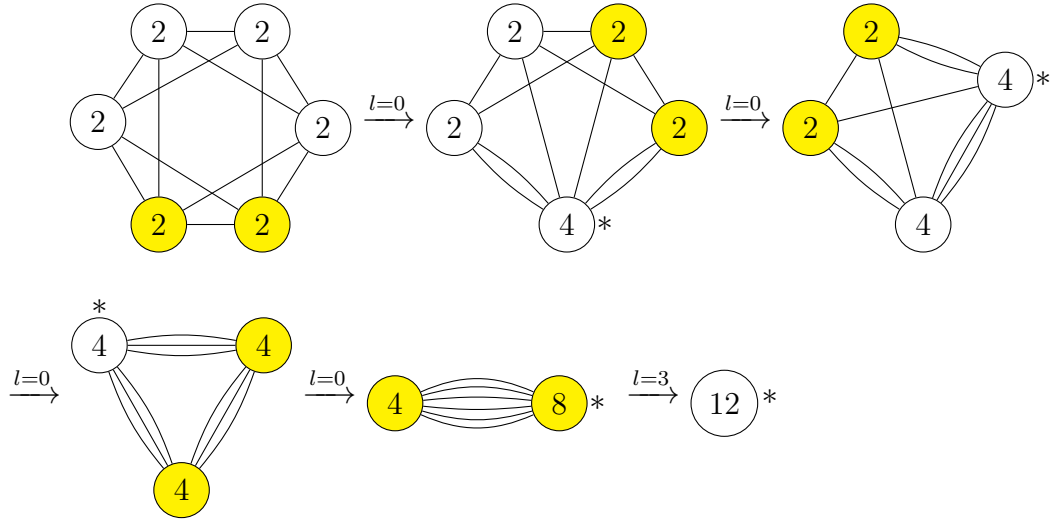
\begin{flushleft}
\qquad
\begin{tikzpicture}[scale=0.8, baseline=(current bounding box.center)]
\node[circle,fill=yellow,draw] (a) at (1,0){$2$};
\node[circle,fill=yellow,draw] (b) at (2.7,0){$2$};
\node[circle,draw] (c) at (3.7,1.6){$2$};
\node[circle,draw] (d) at (2.7,3.2){$2$};
\node[circle,draw] (e) at (1,3.2){$2$};
\node[circle,draw] (f) at (0,1.7){$2$};
\draw (a) edge (b);
\draw (a) edge (c);
\draw (a) edge (e);
\draw (a) edge (f);
\draw (b) edge (c);
\draw (b) edge (d);
\draw (b) edge (f);
\draw (c) edge (d);
\draw (c) edge (e);
\draw (d) edge (e);
\draw (d) edge (f);
\draw (e) edge (f);
\end{tikzpicture}
$\xrightarrow{l=0}$
\begin{tikzpicture}[scale=0.8, baseline=(current bounding box.center)]
\node[circle,draw,label={[shift={(0.5,-0.6)}]$*$}] (a) at (1.85,0){$4$};
\node[circle,fill=yellow,draw] (c) at (3.7,1.6){$2$};
\node[circle,fill=yellow,draw] (d) at (2.7,3.2){$2$};
\node[circle,draw] (e) at (1,3.2){$2$};
\node[circle,draw] (f) at (0,1.6){$2$};
\draw (a) edge[bend left=8] (c);
\draw (a) edge (e);
\draw (a) edge[bend left=8] (f);
\draw (a) edge[bend right=8] (c);
\draw (a) edge (d);
\draw (a) edge[bend right=8] (f);
\draw (c) edge (d);
\draw (c) edge (e);
\draw (d) edge (e);
\draw (d) edge (f);
\draw (e) edge (f);
\end{tikzpicture}
$\xrightarrow{l=0}$
\begin{tikzpicture}[scale=0.8, baseline=(current bounding box.center)]
\node[circle,draw] (a) at (1.85,0){$4$};
\node[circle,draw,label={[shift={(0.5,-0.6)}]$*$}] (c) at (3.2,2.4){$4$};
\node[circle,fill=yellow,draw] (e) at (1,3.2){$2$};
\node[circle,fill=yellow,draw] (f) at (0,1.6){$2$};
\draw (a) edge[bend left=10] (c);
\draw (a) edge (e);
\draw (a) edge[bend left=8] (f);
\draw (a) edge[bend right=10] (c);
\draw (a) edge (c);
\draw (a) edge[bend right=8] (f);
\draw (c) edge[bend left=8] (e);
\draw (c) edge[bend right=8] (e);
\draw (c) edge (f);
\draw (e) edge (f);
\end{tikzpicture}
\end{flushleft}

\begin{flushleft}
$\xrightarrow{l=0}$
\begin{tikzpicture}[scale=0.8, baseline=(current bounding box.center)]
\node[circle,fill=yellow,draw] (a) at (1.85,0){$4$};
\node[circle,fill=yellow,draw] (c) at (3.2,2.4){$4$};
\node[circle,draw,label={[shift={(0,-0.1)}]$*$}] (e) at (0.5,2.4){$4$};

\draw (a) edge[bend left=10] (c);
\draw (a) edge (c);
\draw (a) edge[bend left=10] (e);
\draw (a) edge[bend right=10] (c);
\draw (a) edge (e);
\draw (a) edge[bend right=10] (e);
\draw (c) edge[bend left=10] (e);
\draw (c) edge[bend right=10] (e);
\draw (c) edge (e);
\end{tikzpicture}
$\xrightarrow{l=0}$
\begin{tikzpicture}[scale=0.8, baseline=(current bounding box.center)]
\node[circle,fill=yellow,draw] (c) at (0,0){$4$};
\node[circle,fill=yellow,draw,label={[shift={(0.5,-0.6)}]$*$}] (e) at (2.7,0){$8$};

\draw (c) edge[bend left=6] (e);
\draw (c) edge[bend right=6] (e);
\draw (c) edge[bend left=16] (e);
\draw (c) edge[bend right=16] (e);
\draw (c) edge[bend left=24] (e);
\draw (c) edge[bend right=24] (e);
\end{tikzpicture}
$\xrightarrow{l=3}$
\begin{tikzpicture}[scale=0.8, baseline=(current bounding box.center)]
\node[circle,draw,label={[shift={(0.6,-0.6)}]$*$}] (e) at (3,0){$12$};
\end{tikzpicture}
\end{flushleft}
\medskip

 A successive admissible contraction of $K_3$ is illustrated as below:
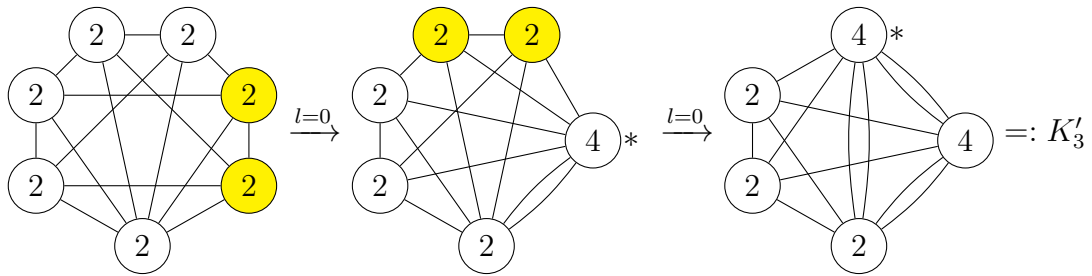
\begin{flushleft}
\qquad
\begin{tikzpicture}[scale=0.8, baseline=(current bounding box.center)]
\node[circle,draw] (a) at (1.75,0){$2$};
\node[circle,draw] (h) at (0,1){$2$};
\node[circle,draw] (g) at (0,2.5){$2$};
\node[circle,fill=yellow,draw] (c) at (3.5,1){$2$};
\node[circle,fill=yellow,draw] (d) at (3.5,2.5){$2$};
\node[circle,draw] (f) at (1,3.5){$2$};
\node[circle,draw] (e) at (2.5,3.5){$2$};
\draw (a) edge (d);
\draw (a) edge (f);
\draw (a) edge (h);
\draw (a) edge (c);
\draw (a) edge (e);
\draw (a) edge (g);
\draw (c) edge (d);
\draw (c) edge (f);
\draw (c) edge (h);
\draw (d) edge (e);
\draw (d) edge (g);
\draw (e) edge (f);
\draw (e) edge (h);
\draw (f) edge (g);
\draw (g) edge (h);
\end{tikzpicture}
$\xrightarrow{l=0}$
\begin{tikzpicture}[scale=0.8, baseline=(current bounding box.center)]
\node[circle,draw] (a) at (1.75,0){$2$};
\node[circle,draw] (h) at (0,1){$2$};
\node[circle,draw] (g) at (0,2.5){$2$};
\node[circle,draw,label={[shift={(0.5,-0.6)}]$*$}] (c) at (3.5,1.75){$4$};
\node[circle,fill=yellow,draw] (f) at (1,3.5){$2$};
\node[circle,fill=yellow,draw] (e) at (2.5,3.5){$2$};
\draw (a) edge (f);
\draw (a) edge (h);
\draw (a) edge[bend left=8] (c);
\draw (a) edge[bend right=8] (c);
\draw (a) edge (e);
\draw (a) edge (g);
\draw (c) edge (f);
\draw (c) edge (h);
\draw (c) edge (e);
\draw (c) edge (g);
\draw (e) edge (f);
\draw (e) edge (h);
\draw (f) edge (g);
\draw (g) edge (h);
\end{tikzpicture}
$\xrightarrow{l=0}$
\begin{tikzpicture}[scale=0.8, baseline=(current bounding box.center)]
\node[circle,draw] (a) at (1.75,0){$2$};
\node[circle,draw] (h) at (0,1){$2$};
\node[circle,draw] (g) at (0,2.5){$2$};
\node[circle,draw] (c) at (3.5,1.75){$4$};
\node[circle,draw,label={[shift={(0.5,-0.6)}]$*$}] (e) at (1.75,3.5){$4$};
\draw (a) edge[bend left=8] (e);
\draw (a) edge[bend right=8] (e);
\draw (a) edge (h);
\draw (a) edge[bend left=8] (c);
\draw (a) edge[bend right=8] (c);
\draw (a) edge (g);
\draw (c) edge[bend left=8] (e);
\draw (c) edge (h);
\draw (c) edge[bend right=8] (e);
\draw (c) edge (g);
\draw (e) edge (h);
\draw (e) edge (g);
\draw (g) edge (h);
\end{tikzpicture}
$=:K_3'$
\end{flushleft}
It is clear that $K_1$ is a spanning submultigraph of $K_3'$. Since $K_1$ is admissibly contractible, so is $K_3'$. Hence $K_3$ is also admissibly contractible.
\medskip

An admissible contraction of $K_4$ is illustrated as below:
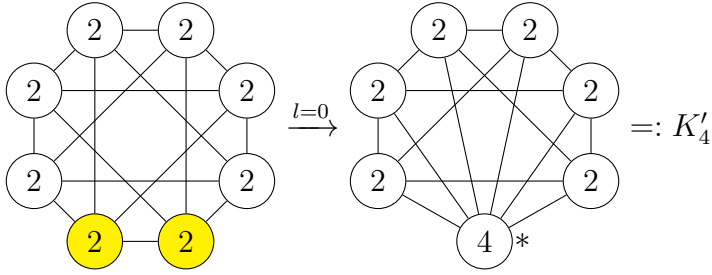
\begin{flushleft}
\qquad
\begin{tikzpicture}[scale=0.8, baseline=(current bounding box.center)]
\node[circle,fill=yellow,draw] (a) at (1,0){$2$};
\node[circle,fill=yellow,draw] (b) at (2.5,0){$2$};
\node[circle,draw] (h) at (0,1){$2$};
\node[circle,draw] (g) at (0,2.5){$2$};
\node[circle,draw] (c) at (3.5,1){$2$};
\node[circle,draw] (d) at (3.5,2.5){$2$};
\node[circle,draw] (f) at (1,3.5){$2$};
\node[circle,draw] (e) at (2.5,3.5){$2$};
\draw (a) edge (b);
\draw (a) edge (d);
\draw (a) edge (f);
\draw (a) edge (h);
\draw (b) edge (c);
\draw (b) edge (e);
\draw (b) edge (g);
\draw (c) edge (d);
\draw (c) edge (f);
\draw (c) edge (h);
\draw (d) edge (e);
\draw (d) edge (g);
\draw (e) edge (f);
\draw (e) edge (h);
\draw (f) edge (g);
\draw (g) edge (h);
\end{tikzpicture}
$\xrightarrow{l=0}$
\begin{tikzpicture}[scale=0.8, baseline=(current bounding box.center)]
\node[circle,draw,label={[shift={(0.5,-0.6)}]$*$}] (a) at (1.75,0){$4$};
\node[circle,draw] (h) at (0,1){$2$};
\node[circle,draw] (g) at (0,2.5){$2$};
\node[circle,draw] (c) at (3.5,1){$2$};
\node[circle,draw] (d) at (3.5,2.5){$2$};
\node[circle,draw] (f) at (1,3.5){$2$};
\node[circle,draw] (e) at (2.5,3.5){$2$};
\draw (a) edge (d);
\draw (a) edge (f);
\draw (a) edge (h);
\draw (a) edge (c);
\draw (a) edge (e);
\draw (a) edge (g);
\draw (c) edge (d);
\draw (c) edge (f);
\draw (c) edge (h);
\draw (d) edge (e);
\draw (d) edge (g);
\draw (e) edge (f);
\draw (e) edge (h);
\draw (f) edge (g);
\draw (g) edge (h);
\end{tikzpicture}
$=:K_4'$
\end{flushleft}
It is clear that $K_3$ is a spanning submultigraph of $K_4'$. Since $K_3$ is admissibly contractible, so is $K_4'$. Hence $K_4$ is also admissibly contractible.
\end{proof}

\begin{lem} \label{contgraph1}
 Let $G$ be a vertex-weighted multigraph. Let $H$ be a submultigraph of $G$.
 Assume the following conditions:
 \begin{enumerate}[label=(\alph*)]
  \item $G$ is completely multipartite;
  \item $\wt_G\geq 2$;
  \item If $\{v_1,\cdots, v_s\}\subset V_H$ is a set of mutually non-adjacent vertices of $H$, then $s\leq \#V_H-4$.
 \end{enumerate}
Then there exists a successive admissible contraction $G'$ of $G$ such that
$H$ is a spanning submultigraph of $G'$.
\end{lem}

\begin{proof}
 We do induction on $q:=\#(V_G\setminus V_H)$.
 If $q=0$, then the lemma is proved by taking $G':=G$. 
 Assume that the lemma is proved for $q-1$. Let $w\in V_G$
 be an arbitrary vertex of $G$. Let $\{v_1,\cdots,v_{s(w)}\}$ be the set of all vertices in $V_H$ that are not adjacent to $w$ in $G$. Since $G$ is completely multipartite, $\{v_1,\cdots,v_{s(w)}\}$ is a set of mutually non-adjacent vertices of $G$ (hence of $H$). By assumption, we have $s(w)\leq \#V_H -4$. This implies that $\rdeg_G(w)\geq 4$ for any 
 vertex $w$ of $G$. Let us pick a vertex $w\in V_G\setminus V_H$, then $w$ is adjacent to a vertex $v\in V_H$. Let $G_1$ be the contraction of $G$ with respect to $\{v,w\}$. Since $\wt_G\geq 2$ and each vertex of $G$ have reduced degree $\geq 4$, $G_1$ is an admissble contraction of $G$ when $l=0$. It is clear that $H$ is a also submultigraph of $G_1$ with $q-1=\#(V_{G_1}\setminus V_H)$, $\wt_{G_1}\geq 2$, and $G_1$ is also a completely multipartite. By the inductive hypothesis, there exists a successive admissible contraction $G_1'$ of $G_1$ such that $H$ is a spanning submultigraph of $G_1'$. The proof is finished by taking $G':=G_1'$.
\end{proof}

\begin{rem}
It is easy to verify that $H=K_i$ satisfies assumption (c) of Lemma \ref{contgraph1} for each $1\leq i\leq 4$. 
\end{rem}

\begin{lem}\label{contgraph2}
Let $G$ be a completely multipartite vertex-weighted multigraph. Assume that for any vertex $v$ of $G$ we have $\rdeg(v)\geq 4$ and
 $\wt(v)\geq 2$.
Then $G$ is admissibly contractible.
\end{lem}

\begin{proof}
 Since $G$ is completely multipartite, there exists a partition
 of vertices $V=\cup_{i=1}^k V_k$ such that two vertices are non-adjacent if and only if they belong to the same $V_i$. Denote $a_i:= \# V_i$. For simplicity we may assume that $a_1\leq a_2\leq \cdots \leq a_k$. Then $\rdeg(v)\geq 4$ implies $\sum_{i=1}^{k-1} a_i\geq 4$. In particular, $k\geq 2$. 
 
 We divide the proof into five cases based on values of $k$ and $a_1$. We will use Lemma \ref{contgraph1} in all cases. Since $G$ satisfies assumptions (a)(b) of Lemma \ref{contgraph1}, we only need to verify assumption (c).
 \smallskip
 
 \textbf{Case 1.} $k\geq 5$.
 
 Let us pick $v_i\in V_i$ for $1\leq i\leq 5$. Let $H$ be the submultigraph
 of $G$ generated by $\{v_1,\cdots,v_5\}$.
 Since $\{v_1,\cdots,v_5\}$ are mutually adjacent in $G$, $K_1$ is a spanning submultigraph of $H$. Hence $H$ satisfies condition (c). By Lemma \ref{contgraph1}, there exists a successive admissible contraction $G'$ of $G$ such that $H$ (hence $K_1$) is a spanning submultgraph of $G'$. By Proposition \ref{samplegraphs}, $K_1$ is admissibly contractible, hence $G$ is admissibly contractible.
 \smallskip
 
 \textbf{Case 2.} $k=4$.
 
 Since $\sum_{i=1}^3 a_i\geq 4$, we have that $a_1,a_2\geq 1$ and $a_3,a_4\geq 2$. Let us pick $v_1\in V_1$, $v_4\in V_2$, $v_2, v_5\in V_3$
 and $v_3, v_6\in V_4$. Let $H$ be the submultigraph of $G$ generated by $\{v_1,\cdots,v_6\}$. It is easy to see that $K_2$ is a spanning submultigraph of $H$, hence $H$ satisfies condition (c). By Lemma \ref{contgraph1}, there exists a successive admissible contraction $G'$ of $G$ such that $H$ (hence $K_2$) is a spanning submultgraph of $G'$. By Proposition \ref{samplegraphs}, $K_2$ is admissibly contractible, hence $G$ is admissibly contractible.
 \smallskip
 
 \textbf{Case 3.} $k=3$ and $a_1\geq 2$.
 
 We know that $a_2,a_3\geq a_1\geq 2$. Let us pick $v_1,v_4\in V_1$, $v_2,v_5\in V_2$ and $v_3,v_6\in V_3$. Let $H$ be the submultigraph of $G$ generated by $\{v_1,\cdots,v_6\}$. It is easy to see that $K_2$ is a spanning submultigraph of $H$, hence $H$ satisfies condition (c). By Lemma \ref{contgraph1}, there exists a successive admissible contraction $G'$ of $G$ such that $H$ (hence $K_2$) is a spanning submultgraph of $G'$. By Proposition \ref{samplegraphs}, $K_2$ is admissibly contractible, hence $G$ is admissibly contractible.
 \smallskip
 
 \textbf{Case 4.} $k=3$ and $a_1=1$.
 
 Since $a_1+a_2\geq 4$, we have $a_2,a_3\geq 3$. Let us pick
 $v_1\in V_1$, $v_2,v_4,v_6\in V_2$ and $v_3,v_5,v_7\in V_3$. Let $H$ be the submultigraph of $G$ generated by $\{v_1,\cdots,v_7\}$. It is easy to see that $K_3$ is a spanning submultigraph of $H$, hence $H$ satisfies condition (c). By Lemma \ref{contgraph1}, there exists a successive admissible contraction $G'$ of $G$ such that $H$ (hence $K_3$) is a spanning submultgraph of $G'$. By Proposition \ref{samplegraphs}, $K_3$ is admissibly contractible, hence $G$ is admissibly contractible.
 \smallskip
 
 \textbf{Case 5.} $k=2$.
 
 Since $a_1\geq 4$, we have $a_1,a_2\geq 4$. Let us pick $v_1,v_3,v_5,v_7\in V_1$ and $v_2,v_4,v_6,v_8\in V_2$. Let $H$ be the submultigraph of $G$ generated by $\{v_1,\cdots,v_8\}$. It is easy to see that $K_4$ is a spanning submultigraph of $H$, hence $H$ satisfies condition (c). By Lemma \ref{contgraph1}, there exists a successive admissible contraction $G'$ of $G$ such that $H$ (hence $K_4$) is a spanning submultgraph of $G'$. By Proposition \ref{samplegraphs}, $K_4$ is admissibly contractible, hence $G$ is admissibly contractible.
\end{proof}

\section{Zaidenberg's method}\label{zai_sec}

\begin{defn}\label{adm_def}
Let $(X,\Delta)$ be a log smooth projective surface pair.
Let $C$ be a reduced curve in $X$. Let $\{\Gamma_t\}_{t\in\bD}$ be a holomorphic flat family of reduced divisors on $X$. Denote by $\Gamma\subset X\times\bD$ the development of $\{\Gamma_t\}_{t\in\bD}$.
We say that $\{\Gamma_t\}_{t\in\bD}$ is an \textit{admissible deformation} of $C$ if $\Gamma_0=C$ and the set $\Gamma_0^{*}:=\{x\in \Gamma_0\mid \Gamma\textrm{ is locally analytically irreducible at }(x,0)\}$ is Brody hyperbolic.
If, moreover, $\Delta+C$ is normal crossing, an admissible deformation $\{\Gamma_t\}_{t\in\bD}$ of $C$ is \textit{nodal} if $\Delta+\Gamma_t$ is normal crossing for any $t\in\bD$. Besides, we say that $\{\Gamma_t^{(j)}\}_{t\in\bD, 1\leq j\leq k}$ is a \emph{successive admissible deformation} of $C$ if for each $1\leq j\leq k$ there exists $t_j\in\bD\setminus\{0\}$, such that $\{\Gamma_t^{(j)}\}_{t\in\bD}$ is an admissible deformation of $\Gamma_{t_{j-1}}^{(j-1)}$ where $\Gamma_{t_0}^{(0)}:=C$.
\end{defn}

The following lemma is a generalization of Zaidenberg's result \cite[Lemma-Definition 3.2]{zai89} to surface pairs. 

\begin{lem}\label{zai}
 Let $(X,\Delta)$ be a log smooth projective 
 surface pair. Let $C$ be a reduced curve in $X$ 
 such that $\Delta+C$ is normal crossing.
 Let $\{\Gamma_t\}_{t\in\bD}$ be an \textit{admissible deformation} of $C$.
 If $X\setminus(C\setminus\Delta)$ is Brody hyperbolic,
 then $X\setminus(\Gamma_t\setminus\Delta)$ is also Brody hyperbolic for any $0<|t|\ll 1$.
 (Note that $X\setminus(C\setminus\Delta)$ being Brody hyperbolic is the same as saying that $X\setminus C$ has \emph{the property of hyperbolic non-percolation through} $C\cap\Delta$ according to \cite{sz02}.)
\end{lem}

\begin{proof}
The proof is similar to \cite[Proof of Lemma-Definition 3.2]{zai89}.
\end{proof}

\begin{lem}\label{smoothing}
 Let $X$ be a smooth projective rational surface. Let $C_1$, $C_2$ be two intersecting rational nodal curves such that $C_1+C_2$ is normal crossing. 
Assume that $(-K_X\cdot C_1)\geq l+1$ and $(-K_X\cdot C_2)\geq l+2$ for some non-negative integer $l< (C_1\cdot C_2)$. Then for any subset $\{x_1,\cdots,x_l\}\subset C_1\cap C_2$, there exists a holomorphic flat family of divisors $\{\Gamma_t\}$ in $X$ such that $\Gamma_0=C_1+C_2$ and $\Gamma_t$ is an irreducible rational nodal curve singular at $x_i$ for any $t\neq 0$ and any $1\leq i\leq l$.
\end{lem}

\begin{proof}
Denote by $\sigma:\tilde{X}=\Bl_{x_1,\cdots,x_l} X\to X$ the blow up of $X$
at $x_1,\cdots,x_l$. Let $E$ be the reduced exceptional divisor of $\sigma$. Let $\tilde{C}_1$ and  $\tilde{C_2}$ be strict transforms of $C_1$ and $C_2$ under $\sigma$. It is easy to see that
$(-K_{\tilde{X}}\cdot \tilde{C}_1)\geq 1$, $(-K_{\tilde{X}}\cdot \tilde{C}_2)\geq 2$ and $(\tilde{C}_1\cdot\tilde{C}_2)>0 $.
It is clear that both $\tilde{C}_1$ and $\tilde{C}_2$ are irreducible rational nodal curves intersecting each other transversally. Denote by $f_i:\bP^1\to \tilde{X}$ the normalization of $\tilde{C}_i$. Since $f_i$ is an immersion, we have an exact sequence $0\to  T_{\bP^1}\to f_i^* T_{\tilde{X}}\to N_{\tilde{C}_i/\tilde{X}}\to 0$, where $\deg  N_{\tilde{C}_i/\tilde{X}}= (-K_{\tilde{X}}\cdot\tilde{C}_i)-2\geq i-2$. Hence $f_1^*T_{\tilde{X}}\otimes\cO(1)$ is nef and $f_2^*T_{\tilde{X}}\otimes\cO(1)$ is ample.
Denote by $f:\bP^1\vee\bP^1\to \tilde{X}$ the gluing morphism of $f_1$ and $f_2$ at an intersection point of $\tilde{C}_1$ and $\tilde{C}_2$. Then $H^1(\bP^1\vee\bP^1,f^* T_{\tilde{X}})=0$ by \cite[II.7.5]{kol96}, so the deformation of $f$ is unobstructed. By \cite[I.2.17]{kol96} there exists a holomorphic flat family of divisors $\{\tilde{\Gamma}_t\}_{t\in\bD}$ such that $\tilde{\Gamma}_0=\tilde{C}_1+\tilde{C}_2$ and $\tilde{\Gamma}_t$ is an irreducible rational nodal curve whenever $t\neq 0$. After a reparametrization of $t$ if necessary we may also assume that $\tilde{\Gamma}_t+E$ is normal crossing for each $t$. The lemma is proved by taking $\Gamma_t:=\sigma_*(\tilde{\Gamma}_t)$.
\end{proof}

\begin{lem}\label{primitive}
 Let $(X,\Delta)$ be a log smooth projective surface
 pair with $X$ rational. Let $C=\sum_{i=1}^m C_i~(m\geq 2)$ be a reduced divisor on $X$ such that each $C_i$ is an irreducible nodal
 rational curve and $\Delta + C$ is normal crossing.
Assume
 \begin{itemize}
 \item $(C_1\cdot C_2)>0$;
  \item $(C_i\cdot (C-C_i))\geq 3$ for any 
 $3\leq i\leq m$;
 \item There exists a non-negative integer 
 $l<(C_1\cdot C_2)$,
 such that $(-K_X\cdot C_i)\geq l+i$ and
 $(C_i\cdot (C-C_1-C_2))\geq 3-l$ for any $i\in\{1,2\}$.
 \end{itemize}
 Then there exists a nodal admissible deformation
 $\{\Gamma_t\}_{t\in\bD}$ of $C$ such that
 $\Gamma_t=A_t+\sum_{i=3}^m C_i$ where $A_t$ is an irreducible rational nodal curve whenever
 $t\neq 0$.
\end{lem}

\begin{proof}
 Let us pick $l$ distinct points $x_1,\cdots,x_l$
 in $C_1\cap C_2$.
 By Lemma \ref{smoothing}, there exists a holomorphic
 flat family $\{A_t\}_{t\in\bD}$ of reduced divisors
 on $X$ such that $A_0=C_1+C_2$ and $A_t$ is an irreducible
 rational nodal curve singular at $x_1,\cdots,x_l$
 for any $t\in\bD\setminus\{0\}$.
 By Bertini's theorem, after a reparametrization of 
 $t$ we may assume that $\Delta+A_t+\sum_{i=3}^m C_i$ is normal
 crossing for any $t\in \bD$. Let $\Gamma_t:=A_t+\sum_{i=3}^m C_i$,
 then it suffices to show that $\Gamma_0^*$
 is hyperbolic. As a divisor in $X\times\bD$, $\Gamma
 =\cA+\sum_{i=3}^m\cC_i$, where $\cA$ is the development
 of $\{A_t\}_{t\in\bD}$ and $\cC_i:=C_i\times\bD$.
 Thus $\Gamma$ is (analytically) reducible at
 $(x,0)$ if $x\in C_i\cap C_j$
 for some $\{i,j\}\neq \{1,2\}$.
 By Lemma \ref{smoothing}, we know that
 $\cA$ is analytically reducible at $(x_1,0),\cdots,
 (x_l,0)$. Thus we have 
 \[
  \Gamma_0\setminus \Gamma_0^*\supset \{x_1,\cdots,x_l\}
  \cup\big(\cup_{\{i,j\}\neq \{1,2\}} (C_{i}\cap 
  C_{j})\big)=:V.
 \]
Since $\Gamma_0=\sum_{i=1}^mC_i$, we only
need to show that $C_i\setminus V$ is hyperbolic 
for any $1\leq i\leq k$.
For each $i\in\{1,2\}$, 
$\#C_i\cap V=\#\big(\{x_1,\cdots,x_l\}\cup(\cup_{j\geq 3}
(C_i\cap C_j))\big)
=l+(C_i\cdot (C-C_1-C_2))\geq 3$. For each $i\geq 3$,
$\#C_i\cap V=\big(C_i\cdot(C- C_i)\big)
\geq 3$. Hence  $C_i\setminus V$ is hyperbolic
for each $1\leq i\leq k$.
The lemma is proved.\end{proof}

\begin{defn}
Let  $X$ be a smooth projective surface. Let $C=\sum_{i=1}^m C_i$ be a reduced normal crossing divisor
on $X$. The \emph{dual graph} $\cD(C):=(V,E,r,\wt)$ of $C$ is a vertex-weighted multigraph defined as follows:
\begin{itemize}
\item $V:=\{v_1,\cdots,v_m\}$;
\item $E:=\cup_{1\leq i<j\leq m} (C_i\cap C_j)$;
\item For each $p\in E$, $r(p):=\{v_i,v_j\}$ where $\{i,j\}$ is the unique unordered pair with $p\in C_i\cap C_j$;
\item For each $v_i\in V$, $\wt(v_i):=(-K_X\cdot C_i)$.
\end{itemize}
\end{defn}

\begin{lem}\label{dualgraph}
 Let $(X,\Delta)$ be a log smooth projective surface pair with $X$ rational.
 Let $C=\sum_{i=1}^m C_i$ be a reduced divisor such that $C+\Delta$ is normal crossing, and each $C_i$ is an irreducible rational curve. If the dual graph $\cD(C)$ is admissibly contractible, then there exists a successive nodal admissible deformation $\{\Gamma_t^{(j)}\}_{t\in\bD, 1\leq j\leq m-1}$ such that
 $\Gamma_{t_{m-1}}^{(m-1)}$ is an irreducible rational nodal curve. If, moreover,  $X\setminus(C\setminus\Delta)$ is Brody hyperbolic, then 
  $\{\Gamma_t^{(j)}\}_{t\in\bD, 1\leq j\leq m-1}$ can be chosen so that
  $X\setminus(\Gamma_t^{(j)}\setminus\Delta)$ is Brody hyperbolic for any
  $t\in\bD$ and any $1\leq j\leq m-1$.
\end{lem}

\begin{proof}
 The successive nodal admissible deformation $\{\Gamma_t^{(j)}\}_{t\in\bD, 1\leq j\leq m-1}$ can be constructed inductively by a successive admissible contraction of the dual graph $\cD(C)$ using Lemma \ref{primitive}. The hyperbolicity part follows from Lemma \ref{zai} and taking reparametrizations of $t$ if necessary.
\end{proof}

\begin{lem}\label{hypsmooth}
 Let $(X,\Delta)$ be a log smooth projective surface 
 pair with $X$ rational. Let $C$ be an irreducible rational nodal
 curve in $X$ such that $\Delta + C$ is normal 
 crossing.
 If $(-K_X\cdot C)\geq 8$ and $\#\Sing(C)\geq 4$, 
 then  there exists a successive nodal admissible deformation
$\{\Gamma_t^{(j)}\}_{t\in\bD, 1\leq j\leq 2}$ of 
$C$ such that $\Gamma_{t_2}^{(2)}$ is
an irreducible smooth hyperbolic curve.
 If, moreover, $X\setminus(C\setminus\Delta)$ is Brody hyperbolic, then 
  $\{\Gamma_t^{(j)}\}_{t\in\bD, 1\leq j\leq 2}$ can be chosen so that
  $X\setminus(\Gamma_t^{(j)}\setminus\Delta)$ is Brody hyperbolic for any
  $t\in\bD$ and any $1\leq j\leq 2$.
\end{lem}

\begin{proof}
Let us pick two nodes $p_1,p_2$ of $C$. Denote by 
$\sigma:\tilde{X}=\Bl_{p_1,p_2}X\to X$ the blow up 
of $X$ at $p_1,p_2$. Let $E=E_1+E_2$ be the reduced exceptional
divisor of $\sigma$. Let $\tilde{C}\subset\tilde{X}$ be the strict
transform of $C$ under $\sigma$. We claim that
$\tilde{C}$ is base point free in $\tilde{X}$.

Since $\tilde{X}$ is rational, we have $H^1(\tilde{X},\cO_{\tilde{X}})=0$.
Thus the claim is equivalent to saying that $\cO_{\tilde{C}}(\tilde{C})$
is globally generated. Since $(-K_X\cdot C)\geq 8$,
we have 
\[
 (-K_{\tilde {X}}\cdot \tilde{C})= (\sigma^*(-K_X)\cdot \tilde{C})-(E\cdot \tilde{C})
 =(-K_X\cdot C)-4\geq 4.
 \]
By adjunction we have $(-K_{\tilde{X}}\cdot \tilde{C})=
(\tilde{C}^2) - 2\#\Sing(\tilde{C})+2$, so we have
$\deg\nu^*\cO_{\tilde{C}}(\tilde{C})=(\tilde{C}^2)\geq 2\#\Sing(\tilde{C})+2$,
where $\nu:\bP^1\to \tilde{C}$ is the normalization
of $\tilde{C}$. Hence the global sections of 
$\nu^*\cO_{\tilde{C}}(\tilde{C}))$ separate any $2\#
\Sing(\tilde{C})+1$ points on $\bP^1$. In particular, this implies that
$\cO_{\tilde{C}}(\tilde{C})$ is globally generated.

Now we have shown that $\tilde{C}$ is base point free
on $\tilde{X}$. By Bertini's theorem, there exists a holomorphic
flat family of irreducible divisors $\{\tilde{\Gamma}_t^{(1)}\}_{t\in\bD}$
on $\tilde{X}$ such that $\tilde{\Gamma}_0^{(1)}=\tilde{C}$
and $(\tilde{X}, \tilde{\Gamma}_t^{(1)}+E+\sigma^*\Delta)$ is log smooth for any $t\in \bD\setminus\{0\}$.
Let $\Gamma_t^{(1)}:=\sigma_*\tilde{\Gamma}_t^{(1)}$.
Since $\tilde{\Gamma}_0^{(1)}=\tilde{C}$ intersects $E_i$
transversally at two points for any $i\in\{1,2\}$,
it is clear that $\tilde{\Gamma}$ has two analytic branches intersecting
$E_i\times\{0\}$ in different points. Thus
$\Gamma^{(1)}$ has two analytic branches at $(p_i,0)$ for each
$i\in\{1,2\}$ which implies that $\Gamma_0^{(1),*}\subset
C\setminus\{p_1,p_2\}$ is hyperbolic. Besides, $(\tilde{X},
\tilde{\Gamma}_t^{(1)}+E+\sigma^*\Delta)$ being log smooth implies
that $\Gamma_t^{(1)}$ is nodal at $p_1,p_2$, smooth
elsewhere and intersects transversally with $\Delta$
for any $t\in\bD\setminus\{0\}$.
Hence $\{\Gamma_t^{(1)}\}_{t\in\bD}$ is a nodal admissible
deformation of $C$ with $\Delta+\Gamma_t^{(1)}$ being
normal crossing for each $t\in\bD$.

Now let us fix an arbitrary $t_1\in\bD\setminus\{0\}$.
Since $p_a(\tilde{C})=\#\Sing(\tilde{C})=\#\Sing(C)-2\geq 2$, 
we know that $\Gamma_t^{(1)}$ is hyperbolic for any 
$t\in\bD\setminus\{0\}$. As we argued before
in showing the base-point-freeness of $\tilde{C}$,
$(-K_X\cdot C)\geq 8\geq 4$ also implies that $C$ is 
base point free on $X$. Hence by Bertini's theorem
there exists a holomorphic flat family of irreducible
divisors $\{\Gamma_t^{(2)}\}_{t\in\bD}$ on $X$ such
that $\Gamma_0^{(2)}=\Gamma_{t_1}^{(1)}$ and $(X,\Gamma_t^{(2)}
+\Delta)$ is log smooth for any $t\in\bD\setminus\{0\}$.
Besides, $\Gamma_0^{(2),*}\subset\Gamma_0^{(2)}=
\Gamma_{t_1}^{(1)}$ is hyperbolic. Hence $\{\Gamma_t^{(2)}\}_{t\in\bD}$
is a nodal admissible deformation of $\Gamma_{t_1}^{(1)}$
such that $\Delta+\Gamma_t^{(2)}$ is normal crossing
for any $t\in\bD$. Besides, $g(\Gamma_t^{(2)})=
p_a(\Gamma_0^{(2)})\geq p_g(\Gamma_0^{(2)})\geq 2$ for any
$t\in\bD\setminus\{0\}$, hence $\Gamma_t^{(2)}$ is hyperbolic
for any $t\in\bD\setminus\{0\}$. The lemma is proved
by taking arbitrary $t_2\neq 0$.
\end{proof}

\begin{lem}\label{zai_gen}
Let $(X,\Delta)$ be a log smooth projective surface pair
with $X$ rational. Let $C=\sum_{i=1}^m C_i$ be a reduced divisor
on $X$ such that $C+\Delta$ is normal crossing. Assume that each $C_i$ is a base-point-free irreducible rational curve with
$(-K_X\cdot C_i)\geq 2$, and it intersects
with at least four other $C_j$'s.
If $X\setminus\big((\cup_{i=1}^m C_i)
\setminus \Delta \big)$ is Brody hyperbolic,
then there exists an irreducible smooth curve $C'$ linearly
equivalent to $\sum_{i=1}^m C_i$ such that both $C'$ and
$X\setminus (C'\setminus\Delta)$ are Brody hyperbolic.
\end{lem}

\begin{proof}
 Let $G:=\cD(C)$ be the dual graph of $C$. Since each $C_i$ is base-point-free, $G$ is completely multipartite. By assumptions, for each vertex $v$ of $G$ we have $\rdeg(v)\geq 4$ and
 $\wt(v)\geq 2$. Hence Lemma \ref{contgraph2} implies that $\cD(C)$ is admissibly contractible. By Lemma \ref{dualgraph}, there exists a successive nodal admissible deformation $\{\Gamma_t^{(j)}\}_{t\in\bD,1\leq j\leq m-1}$ of $C$ such that $\Gamma_{t_{m-1}}^{(m-1)}$ is an irreducible rational curve and $X\setminus(\Gamma_{t_{m-1}}^{(m-1)}\setminus\Delta)$ is Brody hyperbolic. Since each $C_i$ intersects with at least four other $C_j's$, we have $m\geq 5$, hence $(-K_X\cdot \Gamma_{t_{m-1}}^{(m-1)})=\sum_{i=1}^{m}(-K_X\cdot C_i)\geq 10$.
 Since $2\sum_{1\leq i<j\leq m}(C_i\cdot C_j)=\sum_{i=1}^{m}\sum_{j\neq i}(C_i\cdot C_j)\geq 4m$, we have
 \begin{align*}
 \#\Sing(\Gamma_{t_{m-1}}^{(m-1)})& =p_a(\Gamma_{t_{m-1}}^{(m-1)})
 =p_a(C) =\#\Sing(C)-(m-1)\\
 &\geq \sum_{1\leq i<j\leq m}(C_i\cdot C_j) -(m-1)
  \geq m+1\geq 6.   
 \end{align*}
 By applying Lemma \ref{hypsmooth} to $C:=\Gamma_{t_{m-1}}^{(m-1)}$, we know that there exists a successive nodal admissible deformation $\{\Gamma_t^{(j)}\}_{t\in\bD,m\leq j\leq m+1}$ of $\Gamma_{t_{m-1}}^{(m-1)}$ such that $\Gamma_{t_{m+1}}^{(m+1)}$ is an irreducible smooth hyperbolic curve and $X\setminus(\Gamma_{t_{m+1}}^{(m+1)}\setminus\Delta)$ is Brody hyperbolic.
 It is clear that $\Gamma_{t_{m+1}}^{(m+1)}$ is numerically equivalent to $C$, hence they are linearly equivalent since $X$ is rational. The lemma is proved by taking $C':=\Gamma_{t_{m+1}}^{(m+1)}$. 
\end{proof}

The following corollary is a generalization of \cite[Theorem 3.1]{zai89} which says that there exists a smooth plane curve of degree $m$ whose complement is Brody hyperbolic for $m\geq 5$.

\begin{cor}\label{zai_p2}
 Let $m\geq 5$ and $d\geq 4$ be integers.
 Then there exists smooth plane curves $C$ and $S$ of degree $m$ and $d$ respectively, such that 
 $(\bP^2, S + C)$ is log smooth
 and $\bP^2\setminus \big(C \setminus S\big)$ is Brody hyperbolic.
\end{cor}

\begin{proof}
Let $\{C_i\}_{i=1}^m$ be a set of lines in general position
in $\bP^2$. By Corollary \ref{huy_p2}, there exists
a smooth plane curve $S$ of degree $d$ such that
$(\bP^2,\sum_{i=1}^m C_i +S)$ is log smooth and
$\bP^2\setminus\big((\cup_{i=1}^m C_i)\setminus S\big)$
is Brody hyperbolic. We know that
$(C_i^2)=1$ and each $C_i$ intersects all $C_j$'s whenever
$j\neq i$. Since $m-1\geq 4$, the corollary is proved by applying
Lemma \ref{zai_gen} to $(X,\Delta, C_i):=(\bP^2,S, C_i)$.
\end{proof}

The following corollary is related to \cite[1.2]{it15} where they studied 
hyperbolic imbeddedness of $\bF_0\setminus C$.

\begin{cor}\label{zai_fn}
Let $a,b,c,d$ be integers.
Then there exists smooth curves $C$ and $S$ in $\bF_N$
of bidegree $(a,b)$ and $(c,d)$ respectively,
such that $(\bF_N, S + C)$ is log smooth 
and $\bF_N\setminus\big(C\setminus S\big)$
is Brody hyperbolic if one of the
following is true:
\begin{itemize}
 \item $N=0$ and $a,b,c,d\geq 4$;
 \item $N\geq 1$, $a,c\geq 3$ and $b,d\geq 4$.
\end{itemize}
\end{cor}

\begin{proof}
 Let $\{C_i\}_{i=1}^{a+b}$ be a set of curves in $\bF_N$,
 such that $C_i$ is a general curve in $|F|$ for 
any $i\leq a$, and $C_j$ is a general curve in $|T|$
for any $j>a$. By Corollary \ref{huy_fn}, there exists
a smooth curve $S$ of bidegree $(c,d)$ such that
$(\bF_N,\sum_{i=1}^{a+b}C_i+S)$ is log smooth and
$\bF_N\setminus\big((\cup_{i=1}^{a+b} C_i)\setminus S\big)$
is Brody hyperbolic. We know that
$(C_i^2)= 0$ for each $i\leq a$ and $(C_j^2)=N\geq 0$ 
for each $j>a$. From the proof of Corollary \ref{huy_fn}
we know that $C_i$ intersects with at least
four $C_j$'s for each $1\leq i\leq a+b$.
Hence the corollary is proved by applying
Lemma \ref{zai_gen} to $(X,\Delta, C_i):=(\bF_N,S, C_i)$.
\end{proof}

\section{Proofs}\label{proofs}

\begin{lem}\label{finaldef}
 Let $X$ be a smooth projective surface. Let $L$ be a line
 bundle on $X$. Let $n\geq 2$ be an integer.
 Assume that there exists irreducible divisors
 $C\in |L|$ and $S\in |L^{\otimes n}|$ satisfying that
 $(X,S+C)$ is log smooth, and both $C$ 
 and $X\setminus(C\setminus S)$ are Brody hyperbolic.
 Then there exists a smooth curve $D\in |L^{\otimes n}|$ such
 that the degree $n$ cyclic cover of $X$ branched
 along $D$ is Brody hyperbolic.
\end{lem}

\begin{proof}
 
 Let $\{S_t\}_{t\in\bP^1}$ be the linear
 pencil of divisors on $X$ spanned by $S_0:=nC$ and $S_\infty:= S$. Then the development
 of $\{S_t\}$ is an effective Cartier divisor $\cS$ of $X\times \bP^1$.
 Since $S$ and $C$ intersect transversally, it is not hard 
 to check in local charts that $\cS$ is smooth away from the
 finite set $(C\cap S)\times \{0\}$. Let $\pi:\cY\to X\times\bP^1$
 be the degree $n$ cyclic cover of $X\times\bP^1$ branched along $\cS$. Then $\cY$ is smooth
 away from $\pi^{-1}((C\cap S)\times \{0\})$.
 From the construction it is clear that each fiber $Y_t$ of $pr_2\circ\pi:\cY\to\bP^1$
 is a degree $n$ cyclic cover of $X$ branched along $S_t$. 
 Since $S_0=nC$, $Y_0$ is the union of $n$ irreducible 
 components $\{Y_{0,i}\}_{i=1}^n$ such that 
 $Y_{0,i}\cap Y_{0,j}=\pi^{-1}(C\times\{0\})$ for any $i\neq j$,
 and $\pi:(Y_{0,i},\pi^{-1}(C\times\{0\}))\to (X,C)\times \{0\}$ is an isomorphism
 for any $i$.
 
 Assume to the contrary that $Y_{t_n}$ is not Brody hyperbolic for a
 sequence of non-zero complex numbers $(t_n)$ converging to $0$.
 Let $\phi_n:\bC\to Y_{t_n}$ be the sequence of entire
 curves. We may assume that $||\phi_n'(0)||$ tends to infinity after coordinate changes. 
 By Brody Lemma (e.g. \cite{duv}), after choosing a subsequence if necessary, there exists a sequence of reparametrizations $r_n:\bD_{R_n}\to\bD$ where $\lim_{n\to\infty} R_n=+\infty$ such that  $(\phi_n\circ r_n)$ converges to  an entire curve $\phi_\infty:\bC\to Y_0$ as $n\to\infty$.
 Notice that $Y_{0,i}$ is Cartier
 away from  $\pi^{-1}((C\cap S)\times \{0\})$, so Lemma \ref{stabint} implies that $\phi_\infty(\bC)$ is contained in at least one 
 of the $(n+1)$ subsets $\{Y_0^{(i)}\}_{i=0}^n$ of $Y_0$, where
 \begin{align*}
  Y_0^{(0)}&:=\pi^{-1}(C\times\{0\}),\\
  Y_0^{(i)}&:=Y_{0,i}\setminus \pi^{-1}((C\setminus S)\times\{0\})\quad
  \textrm{ for any }1\leq i\leq n.
 \end{align*}
 In particular, at least one of the subsets
 $\{Y_0^{(i)}\}_{i=0}^n$ is not Brody hyperbolic.
Under the projection $\pi$, it is not hard to see that
$Y_0^{(0)}\cong C$ and
$(Y_{0,i}, Y_0^{(i)})\cong (X, X\setminus (C\setminus S))$ 
for any $1\leq i\leq n$.
Thus $Y_0^{(i)}$ is Brody hyperbolic for any $0\leq i\leq n$,
we get a contradiction.

As a result, $Y_t$ is Brody hyperbolic for any $t\neq 0$ sufficiently small.
Since $S_t$ is smooth for general $t$, the lemma is proved by choosing
$D:=S_t$ for $t\neq 0$ sufficiently small.
\end{proof}

\begin{emp}[\textit{Proof of Theorem \ref{main}}] 
Let $\pi: Y\to \bP^2$ be the double cover of $\bP^2$ branched
along $D$.

For the ``only if'' part, if $d\leq 4$ then $Y$ is a
rational surface; if $d=6$ then $Y$ is a K3 surface. In both cases $Y$ is not Brody hyperbolic.
If $d=8$, since Brody hyperbolicity is preserved under small deformation,
we may deform $D$ a bit to ensure that there exists a 
bitangent line $\ell$ of $D$ that meets $D$ transversally in four further points. Hence by Riemann-Hurwitz formula, $\pi^{-1}(\ell)$ is an elliptic curve. Thus $Y$ is never Brody hyperbolic when $d\leq 8$.

For the ``if'' part, Corollary \ref{zai_p2} implies that there exist
plane curves $C$ and $S$ of degree $d/2$ and $d$ respectively,
such that $(\bP^2, C+S)$ is log smooth and $\bP^2
\setminus(C\setminus S)$
is Brody hyperbolic. Since $d/2\geq 5$, $C$ is a smooth curve
of genus at least $6$, so it is Brody hyperbolic.
Thus applying Lemma \ref{finaldef} to $(X,L,n,C,S)
:=(\bP^2, \cO(d/2), 2,C, S)$
finishes the proof.
\qed
\end{emp}

The following theorem is an application of Corollary \ref{zai_p2}
and Lemma \ref{finaldef}. It also improves 
\cite[Theorem 25]{liu16}.

\begin{thm}\label{cycsurf}
 Let $d\geq 10$ be a composite number. Then there exists
 a smooth Brody hyperbolic surface of degree $d$ in $\bP^3$
 that is a cyclic cover of $\bP^2$ under some linear projection.
\end{thm}

\begin{proof}
 By assumption, $d=d_1 d_2$ for some integers $d_1\geq 2$,
 $d_2\geq 5$. Corollary \ref{zai_p2} implies that
 there exist plane curves $C$ and $S$ of degree $d_2$ and $d$
 respectively, such that $(\bP^2, C+S)$ is log 
 smooth and $\bP^2\setminus(C\setminus S)$
 is Brody hyperbolic. Since $d_2\geq 5$, $C$ is a smooth curve
of genus at least $6$, so it is Brody hyperbolic.
 Applying Lemma \ref{finaldef} to $(X,L,n,C,S):=(\bP^2,\cO(d_2),
 d_1,C,S)$ yields that there exists a smooth plane curve $D$
 of degree $d$ such that the degree $d_1$ cyclic cover $Y$
 of $\bP^2$ branched along $D$ is Brody hyperbolic.
 Let $W$ be the degree $d$ cyclic cover of $\bP^2$ branched 
 along $D$, then there is a natural finite surjective morphism 
 $W\to Y$. Since $Y$ is Brody hyperbolic, $W$ is also Brody hyperbolic.
\end{proof}

\begin{emp}[\textit{Proof of Theorem \ref{main2}}] 
Let $\pi: Y\to \bF_N$ be the double cover of $\bF_N$ branched
along $D$.

For the ``only if'' part, assume to the contrary that $b\leq 6$, then $(D\cdot F)=b\leq 6$. Since $\dim |F|=1$, there exists a curve $F_0\in |F|$ such that $F_0$ is tangent to $D$ at some point. As a result, $\pi^{-1}(F_0)$ is a double cover of $\bP^1$ branched along a non-reduced divisor of degree $\leq 6$. This implies that each irreducible component of $\pi^{-1}(F_0)$ is either a rational curve or an elliptic curve, so $Y$ is not Brody hyperbolic. We get a contradiction. Hence we must have $b\geq 8$. If $N=0$, then $a\geq 8$ by the symmetry between $F$ and $T$. If $N\geq 1$, assume to the contrary that $a\leq 4$, then $(D\cdot (T-NF))=a\leq 4$. Let $T'\subset\bF_N$ be the unique curve with negative self-intersection number, then $T'\sim T-NF$. Hence $(D\cdot T')\leq 4$. This implies that each irreducible component of $\pi^{-1}(T')$ is either a rational curve or an elliptic curve, so $Y$ is not Brody hyperbolic. We get a contradiction. Therefore, the proof of the ``only if'' part is completed.

For the ``if'' part, Corollary \ref{zai_fn} implies that there exist
plane curves $C$ and $S$ of bidegree $(a/2,b/2)$ and $(a,b)$ respectively,
such that $(\bF_N, C+S)$ is log smooth and $\bF_N\setminus(C\setminus S)$
is Brody hyperbolic. If $N=0$, then $a,b\geq 8$ implies that $C$ is a smooth curve of genus at least $9$; if $N\geq 1$, then $a\geq 6$ and $b\geq 8$ implies that $C$ is a smooth curve of genus at least $6N+6$. So $C$ is Brody hyperbolic for every $N\geq 0$.
Thus applying Lemma \ref{finaldef} to $(X,L,n,C,S)
:=(\bF_N, \cO_{\bF_N}((a/2)F+(b/2)T), 2, C, S)$
finishes the proof.
\qed
\end{emp}

\begin{rem}\label{finalrmk}
\begin{enumerate}[label=(\alph*)]
    \item According to \cite{hor76}, the canonical model of a Horikawa surface with even $c_1^2$ is either a double cover of $\bP^2$ branched along a degree $8$ or $10$ curve, or a minimal resolution of a double cover of $\bF_N$ branched along a bidegree $(a,6)$ curve where $a$ has finite choices depending on $N$. Hence the ``only if'' parts of Theorem \ref{main} and \ref{main2} imply that a Brody hyperbolic Horikawa surface with even $c_1^2$ has to be a double cover of $\bP^2$ branched along a degree $10$ curve (in fact  one only needs to check algebraic hyperbolicity). However, our deformation method cannot be applied to exhibit other Brody quasi-hyperbolic Horikawa surfaces (i.e. satisfying the Green-Griffiths-Lang conjecture).
    \item Smooth quintic surfaces in $\bP^3$ are natural examples of Horikawa surfaces with odd $c_1^2$. It was shown by Xu \cite{xu94} that a very general quintic surface does not contain any rational or elliptic curve. However, no examples of Brody hyperbolic (even algebraic hyperbolic) quintic surfaces are known so far.
    Notice that the case of a (very) general quintic surface in $\bP^3$ corresponds to the case $d=2n-1$ in the Kobayashi Conjecture (cf. \cite{kob70, kob98}).
    \item Since Brody hyperbolicity is open in the Euclidean topology (see e.g. \cite[3.11.1]{kob98}), Theorem \ref{main} and \ref{main2} imply that there exist non-empty open subsets of certain moduli spaces of double covers of $\bP^2$ or $\bF_N$ that parametrize Brody hyperbolic ones. Besides, we know that Brody hyperbolicity implies algebraic hyperbolicity, and algebraic hyperbolicity is a very generic property in families. Hence Theorem \ref{main} gives an alternative proof of \cite[Theorem 3.2]{rr13}.
    \end{enumerate} 
\end{rem}


\begin{thebibliography}{99}
 \bibitem[Bog77]{bog77} Fedor Bogomolov: {\it Families of curves on a
 surface of general type}. (Russian) Dokl. Akad. Nauk SSSR  236 (1977),
 no. 5, 1041-1044.
 
 \bibitem[Bro78]{bro78} Robert Brody: {\it Compact manifolds and hyperbolicity}. Trans. Amer. Math. Soc. 235 (1978), 213-
219.

 \bibitem[DEG00]{deg00} Jean-Pierre Demailly and Jawher El Goul:
 {\it Hyperbolicity of generic surfaces of high degree in projective $3$-space}.
 Amer. J. Math. 122 (2000), no. 3, 515-546.
 
 \bibitem[Duf44]{duf44} Jacques Dufresnoy: {\it
 Th\'eorie nouvelle des familles complexes normales.
Applications \`a l'\'etude des fonctions 
alg\'ebro\"ides.} (French) 
Ann. Sci. \'Ecole Norm. Sup. (3) 61, (1944). 1-44. 

\bibitem[Duv04]{duv04} Julien Duval: {\it Une sextique hyperbolique dans $\bP^3(\bC)$}. Math. Ann. 330 (2004), 473-476.
 
 \bibitem[Duv17]{duv} Julien Duval: {\it Around Brody lemma}. Preprint available at \href{http://arxiv.org/abs/1703.01850}
 {\textsf{arXiv:1703.01850}}.
 
 \bibitem[Fuj72]{fuj72} Hirotaka Fujimoto: {\it
On holomorphic maps into a taut complex space}. 
Nagoya Math. J. 46 (1972), 49-61. 

 \bibitem[Gre77]{gre77} Mark Green: {\it The hyperbolicity of the
 complement of $2n+1$ hyperplanes in general position in $\bP^n$ and 
 related results}. Proc. Amer. Math. Soc. 66 (1977), no. 1, 109-113. 
 
 \bibitem[GG79]{gg79} Mark Green and Phillip Griffiths: {\it Two applications
 of algebraic geometry to entire holomorphic mappings}. 
 The Chern Symposium 1979 (Proc. Internat. Sympos., Berkeley, Calif., 1979), 
 pp. 41-74, Springer, New York-Berlin, 1980. 
 
 \bibitem[Hor76]{hor76} Eiji Horikawa: {\it Algebraic surfaces of general type with small $C_1^2$. I}. Ann. of Math. (2) 104 (1976), no. 2, 357-387.
 
 \bibitem[Huy16]{huy16} Dinh Tuan Huynh: {\it 
 Examples of hyperbolic hypersurfaces of low degree in projective spaces}. Int. Math. Res. Not. IMRN 2016, no. 18, 5518-5558. 
 
 \bibitem[IT15]{it15} Atsushi Ito and Yusaku Tiba: {\it Curves in quadric and cubic surfaces whose complements are Kobayashi hyperbolically imbedded}. Ann. Inst. Fourier (Grenoble) 65 (2015), no. 5, 2057-2068. 
 
  \bibitem[Kob70]{kob70} Shoshichi Kobayashi: {\it Hyperbolic manifolds and
  holomorphic mappings}. Pure and Applied Mathematics, 2 Marcel Dekker,
  Inc., New York 1970 ix+148 pp. 
 
 \bibitem[Kob98]{kob98} Shoshichi Kobayashi: {\it
 Hyperbolic complex spaces}. Grundlehren der Mathematischen 
 Wissenschaften [Fundamental Principles of Mathematical Sciences],
 318. Springer-Verlag, Berlin, 1998. xiv+471 pp.
 
 \bibitem[Kol96]{kol96} J\'anos Koll\'ar: {\it Rational curves on algebraic varieties}. Ergebnisse der Mathematik und ihrer Grenzgebiete. 3. Folge. A Series of Modern Surveys in Mathematics [Results in Mathematics and Related Areas. 3rd Series. A Series of Modern Surveys in Mathematics], 32. Springer-Verlag, Berlin, 1996. viii+320 pp.
 
 \bibitem[Lan86]{lan86} Serge Lang: {\it Hyperbolic and Diophantine analysis}.
 Bull. Amer. Math. Soc. (N.S.) 14 (1986), no. 2, 159-205.
 
 \bibitem[Liu16]{liu16} Yuchen Liu: {\it Hyperbolicity of cyclic covers and complements}. To appear in Trans. Amer. Math. Soc., 
 available at \href{http://arxiv.org/abs/1601.07514}
 {\textsf{arXiv:1601.07514}}.
 
 \bibitem[LY90]{ly90} Steven Shin-Yi Lu and Shing-Tung Yau: {\it
 Holomorphic curves in surfaces of general type}. Proc. Nat. Acad. Sci.
USA, 87 (January 1990), 80-82.
 
 \bibitem[McQ98]{mcq98}  Michael McQuillan: {\it Diophantine approximations and 
 foliations}.
 Inst. Hautes \'Etudes Sci. Publ. Math. No. 87 (1998), 121-174.
 

 \bibitem[RR13]{rr13} Xavier Roulleau and Erwan Rousseau: {\it On the
 hyperbolicity of surfaces of general type with small $c_1^2$}.
 J. Lond. Math. Soc. (2) 87 (2013), no. 2, 453-477.
 
  \bibitem[SZ02]{sz02} Bernard Shiffman and Mikhail Zaidenberg: {\it Constructing low degree hyperbolic surfaces in $\bP^3$}. Special issue for S. S. Chern. Houston J. Math. 28 (2002), no. 2, 377-388. 
 

 \bibitem[Xu94]{xu94}Geng Xu: {\it Subvarieties of general hypersurfaces in projective space}. J. Differential Geom. 39 (1994), no. 1, 139-172.

 \bibitem[Zai89]{zai89} Mikhail Zaidenberg: {\it Stability of hyperbolic embeddedness and the construction of examples.}
 (Russian) Mat. Sb. (N.S.) 135(177) (1988), no. 3, 361--372, 415; translation in Math. USSR-Sb. 63 (1989), no. 2, 351-361.

\end{thebibliography}
\end{document}